\documentclass[11pt]{amsproc}
\usepackage{mathrsfs}
\usepackage{stmaryrd}
\usepackage{cases}
\usepackage{amssymb}
\usepackage{amsmath}
\usepackage{amsfonts}
\usepackage{graphicx}
\usepackage{amsmath,amstext,amsbsy,amssymb, color}

\newtheorem{theorem}{Theorem}[section]

\newtheorem{proposition}[theorem]{Proposition}

\theoremstyle{definition}

\theoremstyle{remark}

\numberwithin{equation}{section} \errorcontextlines=0

\newcommand{\diag}{\mbox{diag}}

\newcommand{\Pf}{\mathrm{Pf}}
\newcommand{\Hf}{\mathrm{Hf}}
\newcommand{\cdet}{\mathrm{cdet}}
\newcommand{\rdet}{\mathrm{rdet}}
\newcommand{\per}{\mathrm{per}}
\newcommand{\adj}{\mathrm{adj}}
\newcommand{\ot}{\otimes}
\newcommand{\id}{\mathrm{id}}

\begin{document}

\title[Quantum permanents and Hafnians via Pfaffians]
{Quantum permanents and Hafnians via Pfaffians}
\author{Naihuan Jing}
\address{NJ: Department of Mathematics, North Carolina State University, Raleigh, NC 27695, USA}
\email{jing@math.ncsu.edu}
\author{Jian Zhang}
\address{JZ: Institute of Mathematics and Statistics, University of Sao Paulo, Sao Paulo, Brazil 05315-970
}
\email{j.zhang1792@gmail.com}
\thanks{{\scriptsize
\hskip -0.6 true cm MSC (2010): Primary: 17B37; Secondary: 58A17, 15A75, 15B33, 15A15.
Supported by NSFC (grant nos. 11271138 and 11531004), Simons Foundation (grant 198129) and a CSC fellowship.
}}

\begin{abstract} Quantum determinants and Pfaffians or permanents and Hafnians are introduced on the two
parameter quantum general linear group. Fundamental identities among quantum Pf, Hf, and det are
proved in the general setting. We
show that there are two special quantum algebras among the quantum groups, where
the quantum Pfaffians have integral Laurent polynomials as coefficients.
As a consequence, the quantum Hafnian is computed by a closely related quantum permanent and
identical to the quantum Pfaffian on this special quantum algebra.
\end{abstract}
\maketitle
\section{Introduction}

In revealing the correspondence between the bosonic and fermionic versions of the Wick formula,
Caianiello \cite{C} introduced the Hafnian of a symmetric matrix $C$ of even dimension as
\begin{equation}
\Hf(C)=\frac1{n!}\sum_{\sigma\in \Pi'}c_{\sigma(1)\sigma(2)}c_{\sigma(3)\sigma(4)}\cdots c_{\sigma(2n-1)\sigma(2n)},
\end{equation}
where $\Pi'$ consists of all permutations $\sigma\in S_{2n}$ such that $\sigma(2i-1)<\sigma(2i)$, $i=1, \ldots, n$.
It is known that the Hafnian has several properties similar to the Pfaffian \cite{LT}. We will
add one more identity $\Hf(C)=\per(A)$ for
a closely related matrix $A$ (cf. Proposition \ref{P:per-haf}), which appears to be fundamental as one notices that
$\Hf(C)\neq \sqrt{\per(C)}$ in general.
Here the permanent $\per(C)$ \cite{C} is defined
by changing all the signs to $+1$ in $\det(C)$, and sometimes also referred as the positive determinant \cite{M}.

In \cite{JZ}, we have defined the quantum Pfaffian of an even dimensional matrix $B=(b_{ij})$
with noncommutative entries by
\begin{equation}\label{e:pf1}
\Pf_q(B)=\frac1{[n]_{q^4}!}\sum_{\sigma\in \Pi'}(-q)^{l(\sigma)}b_{\sigma(1)\sigma(2)}b_{\sigma(3)\sigma(4)}\cdots b_{\sigma(2n-1)\sigma(2n)},
\end{equation}
where $l(\sigma)$ is the number of inversions of $\sigma\in S_{2n}$.
One can formally define the $q$-Hafnian of $B$ by replacing $q$ by $-q$ in (\ref{e:pf1}).
In general, quantum Pfaffians and Hafnians are polynomials with coefficients of rational functions in $\mathbb Q(q)$.

Classically the relationship between Pfaffians and Hafnians goes much deeper than formal similarity.
On the enveloping algebras of the orthogonal and symplectic Lie algebras, Molev and Nazarov \cite{MN}
found the interesting correspondence between Pfaffians and Hafnians in the study of
Capelli's identities. The goal of this note
is to show that in the coordinate ring of two-parameter general linear quantum group $\mbox{GL}_{r, s}(2n)=\langle a_{ij},
(\rdet_r)^{\pm1}\rangle$,
the quantum Pfaffians
satisfy the fundamental identities (Theorems \ref{T:identity} and \ref{T:per-haf})
\begin{equation}\label{e:ident}
\begin{aligned}
\Pf_r(A^T\mathbb J_{s^{-1}}A)&=\rdet_{r}(A)\\
=\cdet_{s^{-1}}(A)&=\Pf_{s^{-1}}(A\mathbb J_{r}A^T),
\end{aligned}
\end{equation}
where $\rdet(A)$ (resp. $\cdet(A)$) is the quantum row (resp. column) determinant of $A=(a_{ij})_{2n\times 2n}$ and
\begin{align}
\mathbb J_v&=\diag(\underbrace{J_v, \cdots, J_v}_n), \quad J_v=\begin{pmatrix} 0 & 1\\ -v & 0\end{pmatrix}.
\end{align}

 The identities in \eqref{e:ident} can be viewed as a common lift of
the classical identity between Pfaffian and determinant or Hafnian and permanent. In deriving these identities, we have
generalized the method of \cite{JZ} to use various $q$-forms to give a uniformed and general treatment
of quantum determinants and Pfaffians as well as quantum permanents and Hafnians.

We further show that there are two cases $r=\pm s^{-1}=q$ of $\mbox{GL}_{r, s}$ which have a simplified quantum Pfaffian
over $\mathbb Z[q, q^{-1}]$ thanks to the $q$-Pl\"ucker relation. One special case is the usual
quantum general linear group $\mbox{GL}_q=\mbox{GL}_{q, q^{-1}}$, and the other is the quantum group $\mbox{GL}_{q, -q^{-1}}$. On the second quantum
group both the quantum Pfaffian and Haftnian appear and
\begin{equation}
\Pf_q(A^T\mathbb J_{-q}A)=\Hf_{q}(A\mathbb J_{q}A^T)={\rdet}_q(A).
\end{equation}
This interesting phenomenon suggests that the quantum group $\mbox{GL}_{q, -q^{-1}}$
provides a context where more bosonic and fermionic identities
in quantum multilinear algebra may be found.

\section{Quantum determinants}\label{S:qdet}

\subsection{Quantum semigroup $\mathcal A_{r, s}$}
Let $r, s$ be two fixed generic numbers in the complex field $\mathbb C$. The unital algebra $\mathcal A_{r, s}$ is an associative complex algebra generated by $a_{ij}$, $1\leqslant i, j\leqslant n$ subject to the relations:
\begin{align}\label{relation a1}
&a_{ik}a_{il}=ra_{il}a_{ik}, \\\label{relation a2}
&a_{ik}a_{jk}=s^{-1}a_{jk}a_{ik}, \\\label{relation a3}
&ra_{il}a_{jk}=s^{-1}a_{jk}a_{il}, \\\label{relation a4}
&a_{ik}a_{jl}-a_{jl}a_{ik}=(r-s)a_{il}a_{jk},
\end{align}
where $i<j$ and $k<l$. The algebra $\mathcal A_{r, s}$ is a bialgebra under the comultiplication
$\mathcal A_{r, s}\longrightarrow \mathcal A_{r, s}\otimes \mathcal A_{r, s}$ given by
\begin{equation}
\Delta(a_{ij})=\sum_ka_{ik}\otimes a_{kj},
\end{equation}
and the counit given by $\varepsilon(a_{ij})=\delta_{ij}$, the Kronecker delta symbol. This bialgebra is a two-parameter quantum semigroup
(cf. \cite{JL}) that generalizes
the quantum coordinate ring of the general linear group \cite{FRT, LS}, which is the special case of $r=s^{-1}=q$.

The quantum row-determinant and column-determinant of $A$ are defined as follows.
\begin{align}\label{e:qdet}
\rdet_r(A)&=\sum_{\sigma\in S_n}(-r)^{l(\sigma)}a_{1,\sigma(1)}\cdots a_{n,\sigma(n)},\\ \label{e:qper}
\cdet_{s^{-1}}(A)&=\sum_{\sigma\in S_n}(-s)^{-l(\sigma)}a_{\sigma(1),1}\cdots a_{\sigma(n),n}.
\end{align}
 We will show that both are group-like elements (see \eqref{e:grouplike1}-\eqref{e:grouplike2}):
\begin{align*}
\Delta(\rdet_r(A))&=\rdet_r(A)\otimes \rdet_r(A),\\
\Delta(\cdet_{s^{-1}}(A))&=\cdet_{s^{-1}}(A)\otimes \cdet_{s^{-1}}(A).
\end{align*}
In fact, we have the following result \cite{JL} (announced in \cite{T}). For our purpose, we will
give a different proof using the technique of $q$-forms. 
\begin{theorem}\label{T:identity}
In the bialgebra $\mathcal A_{r, s}$ one has that
$$\rdet_r(A)=\cdet_{s^{-1}}(A),$$
and moreover,
$$\rdet_r(A)a_{ij}=(rs)^{j-i}a_{ij}\rdet_r(A).$$
\end{theorem}

To prove this we
introduce two copies of commuting
quantum exterior algebras. The first one is
$$\Lambda(n)=\mathbb C\langle x_1, \ldots,x_n\rangle /I$$
where $I$ is the ideal
$\left( x_i^2, rx_ix_j+x_jx_i|1\leqslant i<j\leqslant n\right)$ and one writes $x\wedge y=x\otimes y\mod I$. Then we have that
\begin{align}\label{qwedge1}
&x_j \wedge x_i=-rx_i\wedge x_j, \\
&x_i\wedge x_i=0, \label{qwedge2}
\end{align}
where $i<j$. The algebra $\Lambda=\Lambda(n)$ is naturally $\mathbb Z$-graded by $\deg(x_i)=1$ and
the top degree subspace is one-dimensional and spanned by
$x_1\wedge\cdots \wedge x_n$. Assume that ${a_{ij}}$'s commute with
${x_i}$'s. We will simply write the general monomial element $a\otimes x$ as $ax$, where $a\in \mathcal A_{r, s}, x\in\Lambda$.
Let
\begin{equation}
\delta_i=\sum_{j=1}^na_{ij}x_j,
\end{equation}
then
$\delta_i$ also satisfy (\ref{qwedge1})-(\ref{qwedge2}). Therefore
\begin{equation}
\delta_1\wedge\cdots\wedge\delta_n=\rdet_r(A)x_1\wedge\cdots\wedge x_n.
\end{equation}

The second quantum exterior algebra $\Lambda'(n)$ is the unital associative algebra
$\mathbb C\langle y_1, \ldots,y_n\rangle/J$, where $J$ is the ideal
$(y_i^2, s^{-1}y_iy_j+y_jy_i | 1\leqslant i<j\leqslant n)$. Using similar convention
for $x_i$'s, the relations are
\begin{align}\label{qwedge3}
&y_j \wedge y_i=-s^{-1}y_i\wedge y_j, \\
&y_i\wedge y_i=0, \label{qwedge4}
\end{align}
where $1\leqslant i<j\leqslant n$. The algebra $\Lambda'$ is also $\mathbb Z$-graded by $\deg(y_i)=1$
with the top degree 1-dimensional subspace spanned by the single vector $y_1\wedge\cdots\wedge y_n$.
Assume that
${y_i}$'s commute with $a_{ij}$ and set
\begin{equation}
\partial_i=\sum_{j=1}^na_{ji}y_j,
\end{equation}
then
$\partial_i$ satisfy the quantum exterior relations (\ref{qwedge3})-(\ref{qwedge4}).  Subsequently
\begin{equation}
\partial_1\wedge\cdots\wedge \partial_n=\cdet_{s^{-1}}(A)y_1\wedge\cdots\wedge y_n.
\end{equation}

The algebra $\Lambda'$ is a right $\mathcal A_{r, s}$-comodule with the coaction $\rho:\Lambda'\longrightarrow \Lambda'\otimes\mathcal A_{r, s}$ given by
$\rho(y_i)=\partial_{i}$.
The comomule identity  $(\rho\otimes \id)\rho=(\id\ot\Delta)\rho$  
leads to the following.
\begin{equation}\label{e:grouplike1}
\begin{aligned}
&(\rho\otimes\id)\rho(y_1\wedge\cdots\wedge y_n)\\
&=(\rho\otimes\id)\partial_1\wedge\cdots\wedge \partial_n\\
&=(\rho\otimes\id)y_1\wedge\cdots\wedge y_n\ot {\cdet}_{s^{-1}}(A)\\
&=y_1\wedge\cdots\wedge y_n\ot{\cdet}_{s^{-1}}(A)\ot {\cdet}_{s^{-1}}(A)
\end{aligned}
\end{equation}
and
\begin{equation}\label{e:grouplike2}
\begin{aligned}
&(\id\ot\Delta)\rho(y_1\wedge\cdots\wedge y_n)\\
&=(\id\ot\Delta)y_1\wedge\cdots\wedge y_n\ot {\cdet}_{s^{-1}}(A)\\
&=y_1 \wedge \cdots \wedge y_n\ot\Delta({\cdet}_{s^{-1}}(A)).\\
\end{aligned}
\end{equation}
Therefore $$\Delta({\cdet}_{s^{-1}}(A))={\cdet}_{s^{-1}}(A)\ot{\cdet}_{s^{-1}}(A).$$
i.e. ${\cdet}_{s^{-1}}(A)$ is a group-like element.

If we put $x=(x_1, \ldots, x_n)^T$, $y=(y_1, \ldots, y_n)^T$, $\delta=(\delta_1, \ldots, \delta_n)^T$ and
$\partial=(\partial_1, \ldots, \partial_n)^T$£¬ then
\begin{align}
\delta &=Ax, \\
\partial &=A^Ty,
\end{align}
where $A=(a_{ij})$. Therefore $\mathcal A_{r, s}$ can be viewed as a set of $r$-linear and $s$-linear transformations preserving the quantum exterior
algebras (\ref{qwedge1})-(\ref{qwedge2}) respectively. The following is a generalization of Manin's result \cite{Ma} for the one-parameter
quantum group.

\begin{proposition} The entries of the matrix $A$ satisfy the relations (\ref{relation a1})-(\ref{relation a4}) of the quantum algebra $\mathcal A_{r, s}$
 if and only if $\delta_i=\sum_ja_{ij}x_j$
and $\partial_i=\sum_ja_{ji}y_j$ satisfy the relations (\ref{qwedge1})-(\ref{qwedge2})
and (\ref{qwedge3})-(\ref{qwedge4}) respectively.
\end{proposition}

We can now prove Theorem \ref{T:identity}.
Suppose that $\Lambda$, $\Lambda'$ and $A$ commute with each other
and generate a $\mathbb Z\times \mathbb Z$-bigraded algebra $\mathcal A=\langle\mathcal A_{r, s}, \Lambda, \Lambda'\rangle=\mathcal A_{r, s}\ot\Lambda\ot\Lambda'$
with $\mathcal A_{(i,j)}=\mathcal A_{r, s}\ot\Lambda_i\ot\Lambda'_j$, where $\Lambda_i$ (resp. $\Lambda_i'$) is
the degree $i$ subspace of $\Lambda$ (resp.  $\Lambda'$).
Consider the following special linear element $\Phi$ in $\mathcal A_{r, s}\otimes \Lambda\otimes \Lambda'$:
\begin{equation}
\Phi=\sum_{i,j=1}^n a_{ji}x_iy_j=x^TA^Ty \in \mathcal A_{(1, 1)}.
\end{equation}

Let $\omega_i=x_i\partial_i=\sum_{j=1}^n a_{ji}x_iy_j\in\mathcal A_{(1, 1)}$. If follows from
 (\ref{qwedge1})-(\ref{qwedge2}) and the commutation relations of $\partial_i$ that
\begin{align}\label{e:ext1}
&\omega_i \wedge \omega_i=0, \quad i=1, \ldots, n, \\ \label{e:ext2}
&\omega_j \wedge \omega_i=\frac rs \omega_i \wedge w_j, \quad 1\leqslant i<j\leqslant n.
\end{align}
Note that $\Phi=\sum_{i=1}^n\omega_i$. Using (\ref{e:ext1})-(\ref{e:ext2}), we have that
\begin{align*}
\wedge^n\Phi &=(\sum_{\sigma\in S_n}(\frac rs)^{l(\sigma)})\omega_1\wedge\cdots\wedge\omega_n \\
&=[n]_{\frac rs}!(x_1\wedge\cdots\wedge x_n)(\partial_1\wedge\cdots\wedge\partial_n)\\
&=[n]_{\frac rs}!\cdet_{s^{-1}}(A) (x_1\wedge\cdots\wedge x_n)(  y_1\wedge\cdots\wedge y_n),
\end{align*}
where $[n]_v!=[n]_v\cdots [1]_v$, $[n]_v=1+v+\cdots+v^{n-1}$ for any variable $v$.
Similarly we can write $\Phi=\sum_{i=1}^n\omega_i'$ with $\omega_i'=\delta_iy_i=\sum_{j=1}^n a_{ij}x_jy_i$.
It is easy to see that
the elements $\omega_i'$ satisfy the same quantum exterior algebra (\ref{e:ext1})-(\ref{e:ext2}). Thus
\begin{align*}
\wedge^n\Phi &=[n]_{\frac rs}!\omega_1'\wedge\cdots\wedge\omega_n' \\
&=[n]_{\frac rs}!(\delta_1\wedge\cdots\wedge\delta_n)(y_1\wedge\cdots\wedge y_n)\\
&=[n]_{\frac rs}!\rdet_r(A) (x_1\wedge\cdots\wedge x_n)(  y_1\wedge\cdots\wedge y_n).
\end{align*}
Subsequently we get the first identity of Theorem \ref{T:identity}:
$$\rdet_r(A)=\cdet_{s^{-1}}(A).$$

For a pair of $t$ indices $i_1, \ldots, i_t$ and $j_1, \ldots, j_t$,
we define the quantum row-minor determinant $\det_r(A_{j_1 \ldots j_t}^{i_1 \ldots i_t})$
as in (\ref{e:qdet}). Then
\begin{align}\label{e:minor1}
\delta_{i_1}\wedge\cdots\wedge\delta_{i_t}=\sum_{j_1<\cdots<j_t}\rdet_r(A^{i_1\ldots i_t}_{j_1\ldots j_t})x_{j_1}\wedge\cdots\wedge x_{i_t},
\end{align}
which implies that $\det_r(A^{i_1\ldots i_t}_{j_1\ldots j_t})=0$ whenever there are two identical rows.

For indices $1\leqslant i_1<\cdots<i_t, i_{t+1}<\cdots<i_n\leqslant n$,
it is readily seen that
$$\delta_{i_1}\wedge\cdots\wedge\delta_{i_t}\wedge \delta_{i_{t+1}}\wedge\cdots\wedge\delta_{i_n}
=(-r)^{i_1+\cdots+i_t-\frac{t(t+1)}2}\delta_1\wedge\cdots\wedge\delta_n.
$$
Note that $x_j$'s also satisfy the same wedge relations. These
then imply the following Laplace expansion by invoking \eqref{e:minor1}:
\begin{equation*}
\rdet_r(A)=\sum_{\substack{j_1<\cdots<j_t\\ j_{t+1}<\cdots<j_n}}(-r)^{(j_1+\cdots+j_t)-
(i_1+\cdots+i_t)}\rdet_r(A_{j_1 \ldots j_t}^{i_1 \ldots i_t})\rdet_r(A_{j_{t+1} \ldots j_n}^{i_{t+1}\ldots i_n}).
\end{equation*}
In particular, for fixed $i, k$
\begin{align}\label{e:Lap01}
\rdet_r(A)\delta_{ik}=\sum_{j=1}^n(-r)^{j-i}a_{ij}\rdet_r (A^{\hat{k}}_{\hat{j}})=\sum_{j=1}^n(-r)^{i-j}\rdet_r(A^{\hat{k}}_{\hat{j}})a_{ij},
\end{align}
where $\hat{i}$ means the omission of $i$ and here we use it to denote the indices $1, \ldots, i-1, i+1, \ldots, n$ for brevity.

As for the quantum (column) determinant or column-minor, we also have
that $\cdet_{s^{-1}}(A^{i_1\ldots i_t}_{j_1\ldots j_t})=0$ whenever there are two identical columns. The corresponding Laplace expansion
for a permutation $(i_1\ldots i_n)$ of $n$ such that $i_1<\cdots<i_r, i_{r+1}<\cdots<i_n$
\begin{align*}
&\cdet_{s^{-1}}(A)\\
&=\sum_{\substack{j_1<\cdots<j_r\\ j_{r+1}<\cdots<j_n}}(-s)^{(i_1+\cdots+i_r)-(j_1+\cdots+j_r)}
\cdet_{s^{-1}}(A^{j_1 \ldots j_r}_{i_1 \ldots i_r})\cdet_{s^{-1}}(A^{j_{r+1} \ldots j_n}_{i_{r+1}\ldots i_n}).
\end{align*}
In particular, we have that for fixed $i, k$
\begin{equation}\label{e:Lap2}
\begin{aligned}
\cdet_{s^{-1}}(A)\delta_{ik}&=\sum_{j=1}^n(-s)^{i-j}a_{ji}\cdet_{s^{-1}}(A_{\hat{k}}^{\hat{j}})\\
&=\sum_{j=1}^n(-s)^{j-i}\cdet_{s^{-1}}(A^{\hat{j}}_{\hat{k}})a_{ji}.
\end{aligned}
\end{equation}

Like the determinant, the minor row determinant is also equal to the column minor determinant
for any pairs of ordered indices $1\leqslant i_1<\cdots<i_t\leqslant n$ and $1\leqslant j_1<\cdots<j_t
\leqslant n$
\begin{equation}\label{minor}
\rdet_r(A^{i_1\ldots i_t}_{j_1\ldots j_t})=\cdet_{s^{-1}}(A^{i_1\ldots i_t}_{j_1\ldots j_t}).
\end{equation}

Let $D_v=\diag(v, v^2, \ldots, v^n)$, and define the inner automorphism $\tau$ of
$\mbox{Mat}_n(\mathcal{A}_{r, s})$ via conjugation by $D_v$ as follows: for any $C=(c_{ij})\in \mbox{Mat}_n(\mathcal{A}_{r, s})$
\begin{align*}
C^{\tau}_v=D^{-1}_vCD_v=(v^{j-i}c_{ij}).
\end{align*}

Let $\adj(A)=((-1)^{i-j}\rdet_r(A_{\hat{i}}^{\hat{j}}))$ be the quantum adjoint matrix associated to
$A$.
Then the Laplace expansions (\ref{e:Lap01})-(\ref{e:Lap2}) give the first version of the {\it quantum Cramer's rule} in $\mathcal A_{r, s}$.
\begin{proposition}\label{p:cramer1} Over the quantum semigroup $\mathcal A_{r, s}(n)$, one has that
\begin{align}\label{e:cramer1}
A^{\tau}_r\adj(A)&=\adj(A)A^{\tau}_{s^{-1}}=\rdet_r(A)I,\\
A(\adj(A))^{\tau}_{r^{-1}}&=(\adj(A))^{\tau}_{s}A=\rdet_r(A)I,
\end{align}
where $I$ is the identity matrix of size $n$.
\end{proposition}
 Using the quantum Cramer's rule, we have that
\begin{align}\label{e:cramer2}
A^{\tau}_r\rdet_r(A)=A^{\tau}_r\adj(A)A^{\tau}_{s^{-1}}=\rdet_r(A)A^{\tau}_{s^{-1}},
\end{align}
therefore $a_{ij}\det_r(A)=(rs)^{i-j}\det_r(A)a_{ij}$, which is the second identity in Theorem \ref{T:identity}.

\subsection{Quantum group $\mbox{GL}_{r, s}(n)$} The second identity in Theorem \ref{e:ident} implies that
$\rdet_r(A)$ is a regular element in the ring $\mathcal{A}_{r, s}$, therefore we can define
the localization $\mathcal A_{r, s}[{\det}_{r}^{-1}]$, which will be denoted as $\mbox{GL}_{r, s}(n)$.

Although in general $A^{\tau}_r\neq A^{\tau}_{s^{-1}}$, this does not
prevent the quantum Cramer's rule \eqref{e:cramer1} from giving an inverse element for $A$ in $\mbox{Mat}_n(\mbox{GL}_{r, s})$.
In fact, \eqref{e:cramer2} gives the following identity:
\begin{align}
(\rdet_r(A))^{-1}A^{\tau}_r=A^{\tau}_{s^{-1}}(\rdet_r(A))^{-1}.
\end{align}
So the quantum Cramer's rule can be rewritten as
\begin{proposition}\label{p:cramer2} On the quantum group $\mbox{GL}_{r, s}$ one has that
\begin{align}\label{e:cramer3}
(\rdet_r(A))^{-1}A^{\tau}_r\cdot\adj(A)&=\adj(A)\cdot A^{\tau}_{s^{-1}}(\rdet_r(A))^{-1}=I,\\ \label{e:cramer4}
A\cdot\adj(A)^{\tau}_{r^{-1}}(\rdet_r(A))^{-1}&=(\rdet_r(A))^{-1}\adj(A)^{\tau}_{s}\cdot A=I,
\end{align}
or equivalently in $\mathrm{Mat}_n(\mathrm{GL}_{r, s})$
\begin{align}
A^{-1}=\rdet_r(A)^{-1}\adj(A)^{\tau}_{s}=\adj(A)^{\tau}_{r^{-1}}\rdet_r(A)^{-1}.
\end{align}
\end{proposition}
The second identity \eqref{e:cramer4} is obtained from \eqref{e:cramer3} by applying the automorphism $\tau^{-1}$ associated
to $D_{r}^{-1}$ or $D_s$ to both sides respectively.

By defining the antipode
\begin{align*}
S(a_{ij})&=(-s)^{j-i}\cdet_{s^{-1}}(A)^{-1}\cdet_{s^{-1}}(A_{\hat{i}}^{\hat{j}})\\
&=(-r)^{i-j}\rdet_r(A_{\hat{i}}^{\hat{j}})\rdet_r(A)^{-1}
\end{align*}
the bialgebra $\mbox{GL}_{r, s}=\mathcal A_{r, s}[\rdet_r^{-1}]$ becomes a Hopf algebra, thus a quantum group
in the sense of Drinfeld. In fact, $AS(A)=S(A)A=I$ follows from Proposition \ref{p:cramer2}, therefore
$(\id\otimes S)\Delta=(S\otimes \id)\Delta=\varepsilon$.

We remark that the quantum semigroup $\mathcal{A}_{r, s}$ offers new identities to quantum multilinear algebra
studied in \cite{JZ} (see \cite{LT} for several interesting identities in the classical situation;
also see \cite{TT, HH, IW, Z, G} for the usual quantum group $\mbox{GL}_q$).
Moreover, if $r, s$ are indeterminates over the field $F$ of characteristic $0$, then the quantum group
$\mathcal A_{r, s}$ can be defined over the rational field $F(r, s)$, and all results in the paper hold
in that situation.

\section{Quantum Pfaffians and Hafnians}

\subsection{$r$-Pfaffians.} Let $B=(b_{ij})$ be a $2n\times 2n$ square matrix with noncommutative entries,
and let $\mathcal B$ be the associative algebra generated by $b_{ij}, i<j$. Assume that $b_{ij}$ commute with the algebra $\Lambda=\Lambda(2n)$ generated by $x_1, \ldots, x_{2n}$,
and consider the algebra $\mathcal B\ot\Lambda$.

Let $\Omega=\sum_{1\leqslant i<j\leqslant 2n}b_{ij}(x_i\wedge x_j)\in \mathcal B_q\ot\Lambda$. The quantum Pfaffian $\Pf_r(B)$ is defined by
\begin{equation}
\wedge ^n\Omega=[n]_{r^4}!\Pf_r(B)x_1\wedge x_2\wedge\cdots \wedge x_{2n}.
\end{equation}

Explicitly the quantum Pfaffian is given by \cite{JZ}:
\begin{align}
\Pf_r(B)
=\frac{1}{[n]_{r^4}!}\sum_{\sigma\in \Pi'}(-r)^{l(\sigma)}b_{\sigma(1)\sigma(2)}b_{\sigma(3)\sigma(4)}\cdots b_{\sigma(2n-1)\sigma(2n)},
\end{align}
where $\Pi'$ is the set of permutations $\sigma$ of $2n$ such that
$\sigma(1)<\sigma(2), \ldots, \sigma(2n-1)<\sigma(2n)$.

In the quantum coordinate ring $\mathcal A_{r, s}(2n)$, we set for $1\leqslant i, j\leqslant 2n$
\begin{equation}\label{e:quadgen}
b_{ij}=\sum_{m=1}^n \cdet_{s^{-1}}(A^{2m-1, 2m}_{i, j})
=\sum_{m=1}^n(a_{2m-1,i}a_{2m,j}-s^{-1}a_{2m,i}a_{2m-1,j}).
\end{equation}
The matrix $B=(b_{ij})$ can be compactly written as
\begin{equation}
B=A^T\mathbb J_{s^{-1}}A,
\end{equation}
where for any variable $s$
\begin{align}\label{e:mat-j}
\mathbb J_{s^{-1}}=\mathbb J_{s^{-1}}(n)=\diag(\underbrace{J_{s^{-1}}, \cdots, J_{s^{-1}}}_n), \quad J_{s^{-1}}=\begin{pmatrix} 0 &1\\ -s^{-1} & 0\end{pmatrix}.
\end{align}

The special case of the following result at $r=s^{-1}=q$ was proved \cite{JR} under a stronger assumption using
representation theory, and later obtained via the quantum exterior algebra \cite{JZ}. We will
derive the generalization under a weaker assumption.

\begin{theorem} \label{T:det-pf} On the quantum coordinate ring $\mathcal A_{r, s}(2n)$ one has that
$${\Pf}_r(A^T\mathbb J_{s^{-1}}A)=\rdet_r(A),$$
where $A^T\mathbb J_{s^{-1}}A=(b_{ij})$ and $b_{ij}$ are defined by (\ref{e:quadgen}).
\end{theorem}
\begin{proof} Note that for $1\leqslant i<j\leqslant 2n$
$$
b_{ij}=\sum_{m=1}^n \rdet_{r}(A^{2m-1, 2m}_{i, j})
=\sum_{m=1}^n(a_{2m-1,i}a_{2m,j}-ra_{2m-1,j}a_{2m,i}).
$$
With the $b_{ij}$, by the quantum exterior relations
of $x_i$'s it follows that the 2-form $\Omega$ can also be written as
\begin{equation}
\Omega=\delta_1\wedge\delta_2+\delta_3\wedge\delta_4+\cdots+\delta_{2n-1}\wedge\delta_{2n},
\end{equation}
where $\delta_i=\sum_{j=1}^{2n} a_{ij}x_j \ (i =1,2, . . . , 2n)$ and they obey (\ref{qwedge1}) and (\ref{qwedge2}). Subsequently
\begin{align}
&(\delta_{2j-1}\wedge \delta_{2j})(\delta_{2i-1}\wedge \delta_{2i})=r^4 (\delta_{2i-1}\wedge \delta_{2i})(\delta_{2j-1}\wedge \delta_{2j}), \ i<j\\
&(\delta_{2i-1}\wedge \delta_{2i})^2=0.
\end{align}
Therefore
\begin{equation}\label{e:ph1}
\wedge^n\Omega=[n]_{r^4}!\delta_1\wedge\cdots\wedge\delta_{2n}=[n]_{r^4}!\rdet_r(A)x_1\wedge\cdots\wedge x_{2n},
\end{equation}
where the last identity uses the wedge formulation of $\det_r(A)$ in $\mathcal A_{r, s}(2n)$.
Hence $\det_r(A)=\Pf_r(B).$
\end{proof}

\subsection{$s^{-1}$-Pfaffians.}

Similar to the $r$-Pfaffian, we utilize another 2-form to define the
$s^{-1}$-Hafnian in $\mathcal B$. Let
\begin{equation}
\Omega'=\sum_{1\leqslant i<j\leqslant 2n}b_{ij}y_i\wedge y_j,
\end{equation}
where $y_i$ satisfy the quantum exterior relations \eqref{qwedge3}-\eqref{qwedge4} and commute with $b_{ij}$'s.
The $s^{-1}$-Hafnian of $B$ is then defined by
\begin{equation}
\wedge ^n\Omega'=[n]_{s^{-4}}!\Pf_{s^{-1}}(B)y_1\wedge\cdots \wedge y_{2n}.
\end{equation}

Using the quantum exterior relations, we have the explicit formula:
\begin{align}
\Pf_{s^{-1}}(B)
&=\frac{1}{[n]_{s^{-4}}!}\sum_{\sigma\in \Pi'}(-s)^{-l(\sigma)}b_{\sigma(1)\sigma(2)}b_{\sigma(3)\sigma(4)}\cdots b_{\sigma(2n-1)\sigma(2n)}.
\end{align}

We also have the second identity between the quantum Pfaffian and quantum determinant.
\begin{theorem} \label{T:per-haf} In the quantum coordinate ring $\mathcal A_{r, s}(2n)$ we have
$$
{\Pf}_{s^{-1}}(A\mathbb J_{r}A^T)=\cdet_{s^{-1}}(A).
$$
where $A\mathbb J_{r}A^T=(b_{ij}')_{2n\times 2n}$ and for  $ i<j$
$$b_{ij}'=\sum_{m=1}^n\rdet_{r}(A^{i, j}_{2m-1, 2m})=\sum_{m=1}^n(a_{i,2m-1}a_{j,2m}-ra_{i,2m}a_{j,2m-1}).
$$
\end{theorem}
\begin{proof} This is proved by a similar method to that of Theorem \ref{T:det-pf}. In fact, note that
the 2-form $\Omega'$ can be rewritten as
\begin{equation}
\Omega'=\partial_1\wedge\partial_2+\partial_3\wedge\partial_4+\cdots+\partial_{2n-1}\wedge\partial_{2n},
\end{equation}
where $\partial_i=\sum_{j=1}^{2n} a_{ji}x_j (i =1,2, . . . , 2n)$, and $\partial_i$'s obey the
quantum exterior algebra \eqref{e:cramer3}-\eqref{e:cramer4}. Therefore
\begin{align*}
\wedge^n\Omega'&=[n]_{s^{-4}}!\partial_1\wedge\cdots\wedge\partial_{2n}\\
&=[n]_{s^{-4}}!\cdet_{s^{-1}}(A)y_1 \wedge y_2 \wedge \cdots\wedge y_{2n},
\end{align*}
which implies that $\cdet_{s^{-1}}(A)={\Pf}_{s^{-1}}(B').$
\end{proof}

Since $\rdet_r(A)=\cdet_{s^{-1}}(A)$ on the quantum semigroup $\mathcal A_{r, s}(2n)$, we have the following result.
\begin{theorem} For $i<j$,
let $b_{ij}=\sum_{m=1}^n\cdet_{s^{-1}}(A^{2m-1, 2m}_{i, j})$ and
$b_{ij}'=\sum_{m=1}^n\rdet_{r}(A^{i, j}_{2m-1, 2m})$ in the quantum semigroup $\mathcal A_{r, s}(2n)$,
then one has that
\begin{equation*}
{\Pf}_r(B)={\Pf}_{s^{-1}}(B').
\end{equation*}
\end{theorem}

\subsection{Quantum Pfaffians and Hafnians}
When $r=\pm s^{-1}=q$, the quantum Pfaffian has a simplified expression and is
an element of $\mbox{GL}_{q, \pm q^{-1}}$ with coefficients in $\mathbb Z[q, q^{-1}]$.
The case of $r=s^{-1}=q$ reduces to the usual quantum group $\mbox{GL}_q$ and was discussed in details in \cite{JZ},
while $r=-s^{-1}$ gives rise to
another quantum group on which both the quantum Pfaffian and Hafnian appear.
We focus on this special case of $r=-s^{-1}=q$.

Let $\overline{\mathcal B}_q$ be the unital associative algebra generated by
$b_{ij}$, $1\leqslant i<j\leqslant 2n$ subject to the quadratic
relations
\begin{equation}\label{reb}
b_{ij}b_{kl}-qb_{ik}b_{jl}+q^{2}b_{il}b_{jk}
=b_{kl}b_{ij}-q^{-1}b_{jl}b_{ik}+q^{-2}b_{jk}b_{il},
\end{equation}
for $1\leqslant i<j<k<l\leqslant 2n$. We refer them as the {\it $q$-Maya relations} or
{\it $q$-Pl\"ucker relations}, since the special case of $q=1$ can be recast
as a Young diagram identity connected to Maya diagrams \cite{H}.

First of all, using a similar method of \cite{JZ} we have the following simplification of the quantum Pfaffian.

\begin{proposition} On the algebra $\overline{\mathcal B}_q$, the quantum Pfaffian $\Pf_q$ is simply given by
\begin{equation}
\Pf_q(B)=\sum_{\sigma\in \Pi}(-q)^{l(\sigma)}b_{\sigma(1)\sigma(2)}b_{\sigma(3)\sigma(4)}\cdots b_{\sigma(2n-1)\sigma(2n)},
\end{equation}
where $\Pi$ is the set of permutations $\sigma\in S_{2n}$ such
that $\sigma(1)<\sigma(2), \ldots, \sigma(2n-1)<\sigma(2n)$ and $\sigma(1)<\sigma(3)<\cdots<\sigma(2n-1)$.
\end{proposition}

By a direct computation as in \cite{JZ}, the elements $b_{ij}$'s given
in \eqref{e:quadgen} can be shown to satisfy the $q$-Maya relations. We remark that
the $b_{ij}$ defined by \eqref{e:quadgen} in the two-parameter quantum semigroup $\mathcal A_{r, s}$ does not
satisfy the quantum Maya relation unless $r^2s^2=1$.
Moreover, the quantum Pfaffian can also be computed iteratively
as follows. A proof can be similarly given as in \cite{JZ}.
\begin{proposition} On the algebra $\overline{\mathcal B}_q$ we have
\begin{equation}
\Pf_q(B)=\sum_{j=2}^{2n}(-q)^{j-2}\Pf_q(B^{1j}_{1j})\Pf_q(B^{2, \ldots, \hat{j},\ldots,2n}_{2, \ldots, \hat{j},\ldots,2n}).
\end{equation}
\end{proposition}

Let $\overline{\mathcal B}'_q$ be the algebra generated by
$b_{ij}'$ for $1\leqslant i<j\leqslant 2n$ modulo the ideal generated by the $(-q)$-Maya
relations
\begin{equation}\label{reb2}
b_{ij}'b_{kl}'+qb_{ik}'b_{jl}'+q^{2}b_{il}'b_{jk}'
=b_{kl}'b_{ij}'+q^{-1}b_{jl}'b_{ik}'+q^{-2}b_{jk}'b_{il}',
\end{equation}
where $1\leqslant i<j<k<l\leqslant 2n$.

Similar to the quantum determinant and Pfaffian, we introduce the quantum (column)-permanent and Hafnian by
\begin{align}\label{e:qper2}
\per_{q}(A)&=\sum_{\sigma\in S_n}q^{l(\sigma)}a_{\sigma(1),1}a_{\sigma(2),2}\cdots a_{\sigma(n),n}\\ \label{e:qhf2}
\Hf_q(B)
&=\frac{1}{[n]_{q^4}!}\sum_{\sigma\in \Pi'}q^{l(\sigma)}b_{\sigma(1)\sigma(2)}b_{\sigma(3)\sigma(4)}\cdots b_{\sigma(2n-1)\sigma(2n)}.
\end{align}

Similar to the quantum Pfaffian, we can simplify the definition of the quantum Hafnian.
\begin{proposition} On the algebra $\overline{\mathcal B}'_q$, the quantum Hafnian $\Hf_q$ is simply given by
\begin{equation}
\Hf_q(B')=\sum_{\sigma\in \Pi}q^{l(\sigma)}b'_{\sigma(1)\sigma(2)}b'_{\sigma(3)\sigma(4)}\cdots b'_{\sigma(2n-1)\sigma(2n)},
\end{equation}
where $\Pi=\{\sigma\in S_{2n}|\sigma(2i-1)<\sigma(2i), \sigma(1)<\sigma(3)<\cdots<\sigma(2n-1)\}$.
\end{proposition}

Then $\Hf_q$ can be evaluated inductively as follows.
\begin{equation}
\Hf_q(B')=\sum_{j=2}^{2n}q^{j-2}\Hf_q({B'}^{1j}_{1j})\Hf_q({B'}^{2, \ldots,\hat{j},\ldots,2n}_{2, \ldots,\hat{j},\ldots,2n}).
\end{equation}

Moreover, when the matrix elements $b_{ij}'$ are special quadratic elements in $\mbox{GL}_{q, -q^{-1}}$, we can also express $\Hf_q$
in terms of $\per_q$. The following is an immediate consequence of Theorem \ref{T:per-haf}.
\begin{proposition}\label{P:per-haf} In the quantum coordinate ring $\mathcal A_{q, -q^{-1}}(2n)$ we have
\begin{align*}
{\Hf}_{q}(A\mathbb J_{q}A^T)&={\per}_{q}(A),
\end{align*}
where
$A\mathbb J_{q}A^T=B'$ is given by ($i<j$)
$$b_{ij}'=\sum_{m=1}^n\per_{q}(A^{i, j}_{2m-1, 2m})=\sum_{m=1}^n(a_{i,2m-1}a_{j,2m}+qa_{j,2m-1}a_{i,2m}).
$$
\end{proposition}

\bigskip
\centerline{\bf Acknowledgments}
The first author would like to thank Alex Molev for stimulating discussions on related topics.
We also thank Ken Goodearl for his interest on the paper. The authors are grateful to the
referee for encouragement to upgrade the results from the one-parameter quantum group
to the two-parameter analogue.

\vskip 0.1in

\bibliographystyle{amsalpha}

\end{document}